\documentclass{article}

\usepackage{amsmath, enumerate, a4wide}
\usepackage{amsfonts, amsthm}
\usepackage{amssymb}
\usepackage{verbatim}
\usepackage[colorlinks]{hyperref}

\newtheorem{lemma}{Lemma}

\newtheorem{theorem}[lemma]{Theorem}
\newtheorem{definition}[lemma]{Definition}

\title{Graphs with Sudoku number  $n-1$}
\author{Alexey Pokrovskiy  \footnote{Email: \texttt{dr.alexey.pokrovskiy@gmail.com}. Address: Department of Mathematics, University College London, UK.}}
\begin{document}
\maketitle
\begin{abstract}
Recently Lau-Jeyaseeli-Shiu-Arumugam introduced the concept of the ``Sudoku colourings'' of graphs --- partial $\chi(G)$-colourings of $G$ that have a unique extension to a proper $\chi(G)$-colouring of all the vertices. They introduced the Sudoku number of a graph as the minimal number of coloured vertices in a Sudoku colouring. They conjectured that   a connected graph has Sudoku number $n-1$ if, and only if, it is complete. In this note we prove that this is true. 
\end{abstract}

\section{Introduction}
All colourings in this note are vertex-colourings.
A vertex colouring of a graph is \emph{proper} if adjacent vertices receive different colours. The chromatic number, $\chi(G)$ is the minimal number of colours in a proper colouring of $G$. A \emph{partial} proper colouring of a graph is a colouring of some subset of $V(G)$ which doesn't give adjacent vertices the same colour. We say that a colouring $\phi$ extends a partial colouring $\psi$ if all vertices coloured in $\psi$ receive the same colour in $\phi$.
A variety of tasks can be encoded as taking a partial proper colouring in a graph and then extending it to a full proper colouring of all the vertices. 

For example the well known ``Sudoku puzzle'' can be encoded in this form. Consider a graph $G_{\mathrm{Sudoku}}$ on $27$ vertices that are identified with the cells in a $9\times 9$ grid. Add an edge between any two vertices in the same column, between any two vertices in the same row, and between any two vertices in the same $3\times 3$ box. It is easy to see that a proper $9$-colouring of $G_{\mathrm{Sudoku}}$  exactly corresponds to filling in $9\times 9$ array according to the rules of Sudoku puzzles. Thus a Sudoku puzzle can be summarized as ``you are given a partial colouring of $G_{\mathrm{Sudoku}}$ and need to complete it to a proper colouring of all the vertices of $G_{\mathrm{Sudoku}}$''. 

One of the conventions for designing Sudoku puzzles is that there should always be precisely one way of filling in the $9\times 9$ array i.e. there should always exist one solution, and there shouldn't exist multiple solutions. This motivates the definition of a Sudoku colouring of a graph. Lau-Jeyaseeli-Shiu-Arumugam defined it as follows in~\cite{Sudoku_colourings}.
\begin{definition}
A Sudoku colouring of a graph $G$ is a partial proper $\chi(G)$-colouring of $G$ which has precisely one extension to a  $\chi(G)$-colouring of $G$.
\end{definition}
Sudoku colourings of $G_{\mathrm{Sudoku}}$ are thus in one-to-one correspondence with uncompleted Sudoku puzzles (that have unique completion). But we can also investigate Sudoku colourings of general graphs.  

Lau-Jeyaseeli-Shiu-Arumugam~\cite{Sudoku_colourings} defined the Sudoku number of $G$, $sn(G)$ as the smallest number of coloured vertices in a Sudoku colouring of $G$. The motivation for this is that $sn(G_{\mathrm{Sudoku}})$ now asks for the minimum number of clues (i.e. non-blank entries) in a Sudoku puzzle with unique solution. This number has been determined as $sn(G_{\mathrm{Sudoku}})=17$ by McGuire,  Tugemann, and Civario using a computer-assisted proof~\cite{mcguire2014there}. 

For general graphs, Lau-Jeyaseeli-Shiu-Arumugam~\cite{Sudoku_colourings} determined  $sn(G)$ for various classes of graphs $G$ and obtained bounds for other classes. For example, they showed that $sn(G)=1$ if, and only if, $G$ is connected and bipartite. On the other extreme, they showed that $sn(G)\leq |G|-1$ for all graphs and conjectured that for connected graphs, equality holds if, and only if, $G$ is complete. Here we show that this is the case. 
\begin{theorem}\label{Main_Theorem}
A connected graph has $sn(G)=|G|-1$ if, and only if, $G$ is complete.
\end{theorem}
The backwards direction already appears in~\cite{Sudoku_colourings} (see Corollary 3.4), so we focus on proving the statement ``let $G$ be a connected graph with  $sn(G)=|G|-1$. Then $G$ is complete''. This amounts to showing that non-complete graphs have partial proper $\chi(G)$-colourings with $\geq 2$ uncoloured vertices which have a unique extension to a full proper $\chi(G)$-colouring of $G$.

\section{Proofs}
In a $k$-coloured graph our set of colours will always be $[k]=\{1, \dots, k\}$.
 For a  colouring $\phi$ and vertices $u_1, \dots, u_k$, we use $\phi-u_1-\dots-u_k$ to mean the partial colouring formed by uncolouring the vertices $u_1, \dots, u_k$. For a partial colouring $\phi$ and a set of vertices $S$, we define $\phi(S):=\{\phi(s):s\in S\}$ to mean the set of colours appearing in $S$. 
We use $N(v)$ to denote the set of vertices connected to $v$ by an edge (our  graphs are simple, so this never includes $v$ itself).

The following is the tool we use for constructing Sudoku colourings in this note. It gives two kinds of Sudoku colourings with two uncoloured vertices.
\begin{lemma}\label{Lemma_sudoku_colouring}
Let $\psi$ be a partial proper $\chi(G)$-colouring with exactly two uncoloured vertices $u,v$. Suppose that either of the following holds:
\begin{enumerate}[(i)]
\item $uv$ is a nonedge and $|\phi(N(u))|=|\phi(N(v))|=\chi(G)-1$.
\item $uv$ is an edge, $|\phi(N(u))|=\chi(G)-1$, $|\phi(N(v))|=\chi(G)-2$, and $\phi(N(v))\subset  \phi(N(u))$.
\end{enumerate}
Then $\psi$ is a Sudoku colouring.
\end{lemma}
\begin{proof}
In case (i), there is precisely one colour missing from $N(u)$ and precisely one colour missing from $N(v)$. In a proper $\chi(G)$-colouring, $u$ and $v$ must receive exactly these colours, and so the extension is unique.

In case (ii), there are two colours $c,d$ missing from $N(v)$ and one of these colours (say $c$), is missing from $N(u)$. To complete the colouring $u$ must receive colour $c$, and $v$ must receive colour $d$, so the extension is unique.
\end{proof}
The following definition is crucial for us. 
\begin{definition}
Let $\phi$ be a proper $\chi(G)$-colouring of $G$. A vertex is \textbf{full} if it is adjacent to vertices of all colours (aside from its own) i.e. if $\phi(N(v))=[\chi(G)]\setminus \phi(v)$ (or equivalently if $|\phi(N(v))|=\chi(G)-1$).
\end{definition}

In graphs with Sudoku number $|G|-1$, it turns out that the full vertices form a complete subgraph.
\begin{lemma}\label{Lemma_full_vertices_complete}
Let $sn(G)=|G|-1$ and let $\phi$ be a proper $\chi(G)$-colouring of $G$. Then any two full vertices are connected by an edge.
\end{lemma}
\begin{proof}
Let $u,v$ be full and suppose for contradiction that $uv$ is not an edge in $G$. Consider the partial colouring $\psi:=\phi-u-v$. Since $u,v$ are full, we have $|\phi(N(u))|=|\phi(N(v)))|=\chi(G)-1$. Since $uv$ is a non-edge, the neighbours of $u,v$ all remain coloured in $\psi$ and so  $|\psi(N(u))|=|\psi(N(v)))|=\chi(G)-1$. Thus, by Lemma~\ref{Lemma_sudoku_colouring} (i), $\psi$ is a Sudoku colouring. It has two uncoloured vertices, contradicting $sn(G)=|G|-1$.
\end{proof}

We say that a proper $k$-colouring of $G$ is $c$-minimal if it has as few colour $c$ vertices as possible (for a $k$-colouring of $G$). The following lemma shows that there are a lot of full vertices around all $c$-coloured vertices in a $c$-minimal colouring.
\begin{lemma}\label{Lemma_1_minimal_colouring}
Let $\phi$ be a $1$-minimal proper $\chi(G)$-colouring of $G$. Let $v$ be a vertex with $\phi(v)=1$. Then for every colour $c=2, \dots, \chi(G)$, the vertex  $v$ has at least one colour $c$ neighbour $u$ with $u$ full.
In particular, $v$ is full.
\end{lemma}
\begin{proof}
First notice that is impossible that $v$ has no colour $c$ neighbours --- indeed otherwise, we could recolour $v$ with colour $c$ to get a proper colouring with one fewer colour 1 vertex (contradicting $1$-minimality). This proves the ``in particular $v$ is full'' part --- since we've shown that every colour other than $\phi(v)=1$ appears on $N(v)$.

Now let the set of colour $c$ neighbours of $v$ be $\{u_1, \dots, u_k\}$. Suppose for contradiction that none of these are full --- equivalently there are colours $c_1, \dots, c_k \in [\chi(G)]\setminus c$ with $c_i$ missing from $N(u_i)$.
Note that $c_i\neq 1$ for all $i$, since $v\in N(u_i)$ and $\phi(v)=1$. Note that $\{u_1, \dots, u_k\}$ is an independent set since all these vertices have colour $c$ and the colouring is proper.

Now recolour $u_i$ by $c_i$ for each $i$ and recolour $v$ by $c$. Notice that this colouring is proper. To show this, we need to check that the recoloured vertices $v, u_1, \dots, u_k$ have different colours to all their neighbours (everywhere else the colouring remains proper just because $\phi$ was proper). Indeed $v$ has no colour $c$ neighbours since $\{u_1, \dots, u_k\}$ was the set of all colour $c$ neighbours of $v$ and these have all been recoloured by colours $c_i\neq c$. Vertex $u_i$ has no colour $c_i$ neighbour since it initially had no colour $c_i$ neighbours and the only neighbour of $u_i$ that was recoloured was $v$ (which received colour $c\neq c_i$).

But the new colouring we have has one fewer colour $1$ vertex, contradicting $1$-minimality.
\end{proof}

Applying the above lemma to a graph with Sudoku number $|G|-1$ gives even more structure in a minimal colouring.
\begin{lemma}\label{Lemma_1_minimal_sudoku}
Let $sn(G)=|G|-1$ and let $\phi$ be a $1$-minimal proper $\chi(G)$-colouring of $G$. Then there is precisely one colour $1$  vertex. Additionally, this vertex $v$ is full and has $|N(v)|=\chi(G)-1$.
\end{lemma}
\begin{proof}
Let $v$ be a colour $1$ vertex. By Lemma~\ref{Lemma_1_minimal_colouring} $v$ is full. There can't be another colour $1$ vertex $z$ --- because otherwise $z$ would also be full, which would give two disconnected full vertices (contradicting Lemma~\ref{Lemma_full_vertices_complete}).

Suppose that $|N(v)|>\chi(G)-1$. Then, by the pigeonhole principle there must be some colour $c$ which occurs more than once on $N(v)$. By Lemma~\ref{Lemma_1_minimal_colouring} $v$ has some full colour $c$ neighbour $u$. Let $w$ be some other colour $c$ neighbour of $v$. Now let $\psi:=\phi-u-v$. Note that $\psi(N(v))=\phi(N(v))=[\chi(G)]\setminus 1$ (the first equality holds because precisely one neighbour $u$ of $v$ was uncoloured, but that neighbour $u$ had colour $c$ which is still present at $w$. The second equality holds because $v$ is full and has colour $1$). Also  $\psi(N(u))=\phi(N(u))\setminus 1=[\chi(G)]\setminus\{1,c\}$ (the first equality holds because precisely one neighbour $v$ or $u$ was uncoloured, and that neighbour had colour $1$ which isn't present anywhere else in the graph. The second equality holds because $\phi(N(u))=[\chi(G)]\setminus c$ since $u$ is full and has colour $c$). Thus by Lemma~\ref{Lemma_sudoku_colouring} (ii), $\psi$ is a Sudoku colouring with two uncoloured vertices, contradicting $sn(G)=|G|-1$.
\end{proof}

We are ready to prove our main theorem. 

\begin{proof}[Proof of Theorem~\ref{Main_Theorem}]
The backwards direction already appears in~\cite{Sudoku_colourings} (see Corollary 3.4), so it remains to prove that every connected $G$ with  $sn(G)=|G|-1$ is complete. To that end, let $G$ be connected with  $sn(G)=|G|-1$.

Consider a 1-minimal proper $\chi(G)$-colouring of $G$. By Lemma~\ref{Lemma_1_minimal_sudoku}, there is precisely one colour 1 vertex, call it $v$. Also by Lemma~\ref{Lemma_1_minimal_sudoku}, $v$ is full and $|N(v)|=\chi(G)-1$ --- which, using the definition of ``full'', implies that all neighbours of $v$ have different colours. By Lemma~\ref{Lemma_1_minimal_colouring}, we have that all neighbours of $v$ are full. Now we have that all vertices in $v\cup N(v)$  are full, and so Lemma~\ref{Lemma_full_vertices_complete} tells us that $v\cup N(v)$ is complete. 

Unless $G$ is complete, then, by connectedness, there is some vertex $w$ outside $v\cup N(v)$ with a neighbour $u$ in $N(v)$. Now construct a colouring $\psi$ by recolouring $w$ by colour $1$ and uncolouring $u,v$. First notice that this is a proper (partial) colouring --- indeed $w$ has no colour $1$ neighbours since $v$ is the unique colour $1$ vertex and $w\not\in N(v)$. We have that $\psi(N(v))=\phi(N(v))\setminus \phi(u)=[\chi(G)]\setminus\{1, \phi(u)\}$ (the first equality holds because the neighbour $u$ of $v$ was uncoloured and $\phi(u)$ doesn't appear anywhere else on $N(v)$ since the neighbours of $v$ have different colours. The second equality holds because $\phi(N(v))=[\chi(G)]\setminus 1$ since $v$ is full and has colour $1$).  Also $\psi(N(u))=[\chi(G)]\setminus\{\phi(u)\}$ (colour $1$ is present on $N(u)$ in $\psi$ because $\psi(w)=1$ and $w\in N(u)$. All colours in $[\chi(G)]\setminus\{1, \phi(u)\}$ are present on $N(u)$ in $\psi$ because the vertices of $N(v)\setminus u$ have exactly these colours in $\phi$, $u$ is connected to all of them, and their colours don't change going from $\phi$ to $\psi$). Thus by Lemma~\ref{Lemma_sudoku_colouring} (ii), $\psi$ is a Sudoku colouring with two uncoloured vertices, contradicting $sn(G)=|G|-1$.
\end{proof}

\bibliography{Sudoku}
\bibliographystyle{plain}

\end{document}